\documentclass{amsart}
\usepackage{xypic}
\usepackage{amsmath}
\usepackage{amssymb}
\usepackage{amsfonts}
\usepackage{mathtools}
\usepackage{amscd}
\usepackage{url}

\usepackage{tkz-fct}
\usepackage{tikz}
\usepackage{tikz-cd}
\usetikzlibrary{matrix}
\usetikzlibrary{arrows}

 \newtheorem{theorem}{Theorem}[section]
 \newtheorem{corollary}[theorem]{Corollary}
 
 \newtheorem{lemma}[theorem]{Lemma}
 \newtheorem{proposition}[theorem]{Proposition}
 \theoremstyle{definition}
 \newtheorem{definition}[theorem]{Definition}
 \theoremstyle{remark}
 \newtheorem{remark}[theorem]{Remark}
 
 \newtheorem{example}[theorem]{Example}
  
\numberwithin{equation}{section}

\newcommand{\isomto}{\overset{\sim}{\to}}
\newcommand{\longisomto}{\overset{\sim}{\longrightarrow}}

\newcommand{\longto}{\longrightarrow}

\makeatletter
\newcommand*\rel@kern[1]{\kern#1\dimexpr\macc@kerna}
\newcommand*\widebar[1]{%
  \begingroup
  \def\mathaccent##1##2{%
    \rel@kern{0.8}%
    \overline{\rel@kern{-0.8}\macc@nucleus\rel@kern{0.2}}%
    \rel@kern{-0.2}%
  }%
  \macc@depth\@ne
  \let\math@bgroup\@empty \let\math@egroup\macc@set@skewchar
  \mathsurround\z@ \frozen@everymath{\mathgroup\macc@group\relax}%
  \macc@set@skewchar\relax
  \let\mathaccentV\macc@nested@a
  \macc@nested@a\relax111{#1}%
  \endgroup
}
\makeatother

\DeclareMathOperator\Aut{Aut}

\DeclareMathOperator\Hom{Hom}
\DeclareMathOperator\Ell{\mathcal{E}\!\ell\ell}
\DeclareMathOperator\End{End}
\DeclareMathOperator\Spec{Spec}
\DeclareMathOperator\Sh{Sh}

\DeclareMathOperator\coker{coker}
\DeclareMathOperator\Pic{Pic}

\DeclareMathOperator\GL{GL}
\DeclareMathOperator\GSp{GSp}
\DeclareMathOperator\SO{SO}
\DeclareMathOperator\Orth{O}

\DeclareMathOperator\SL{SL}

\DeclareMathOperator\Gal{Gal}

\DeclareMathOperator\Res{Res}

\def\bQ{{\mathbf{Q}}} \def\bZ{{\mathbf{Z}}} 
 \def\bG{{\mathbf{G}}} \def\bR{{\mathbf{R}}}
\def\bC{{\mathbf{C}}}  
\def\bA{{\mathbf{A}}}  \def\bS{{\mathbf{S}}}

 \def\cU{{\mathcal{U}}} 
  \def\cL{{\mathcal{L}}} 
\def\cM{\mathcal M}  
\def\cH{\mathcal H}  
\def\cK{{\mathcal K}}  \def\cA{{\mathcal{A}}}
 \def\cB{{\mathcal B}} \def\cX{{\mathcal{X}}}
\def\cV{{\mathcal V}}

 \def\rR{{\rm R}} \def\rH{{\rm H}}

\def\ab{{\rm ab}}
\def\id{{\rm id}}

\def\et{{\rm et}}

\def\an{{\rm an}}

\def\rec{{\rm rec}}

\begin{document}

%
%
%
%
%
%
%
%
%

\title{Complex multiplication and Shimura stacks}

\author[L. Taelman]{Lenny Taelman}

\address{
Korteweg--de Vries Institute for Mathematics\\
Universiteit van Amsterdam\\
P.O. Box 94248\\
1090 GE Amsterdam\\
the Netherlands
}

\email{l.d.j.taelman@uva.nl}

\begin{abstract} 
We prove a variant of the reciprocity laws for CM abelian varieties, CM K3 surfaces, and CM points on Shimura varieties. Given a CM object over $\bC$, our variation describes the set of all models over a given number field $F$ in terms of associated representations of $\Gal_F$. An essential feature is that we work with Shimura stacks to deal with objects that have non-trivial automorphisms. 

To prove the result on K3 surfaces, we show that the stack of polarized K3 surfaces of given degree is an open substack of a certain Shimura stack. The precise statement of this folklore fact seems to be missing from the literature.
\end{abstract}

\maketitle

\tableofcontents

\section{Introduction and statement of the main results}

\subsection{Abelian varieties}\label{subsec:intro-AV}
Let $A$ be a complex abelian variety.  Let $F$ be a subfield of $\bC$. Then $A$ may or may not be defined over $F$, and if it is defined over $F$ there may be multiple non-isomorphic models $\cA/F$. Let $\cB(A,F)$ be the set of isomorphism classes of models of $A$ over $F$, that is
\[
	\cB(A,F) := \{ (\cA,\alpha) \mid \text{$\cA/F$, $\alpha\colon \cA_\bC \isomto A$}\}/
	\cong
\]
Let $\bar F$ be the algebraic closure of $F$ in $\bC$ and set $\Gal_F := \Gal(\bar F / F)$. The choice of a
model $(\cA,\alpha)$ of $A$ over $F$ defines an action of $\Gal_F$ on $\rH^1(A,\bZ_\ell)$.

\begin{lemma}\label{lemma:injectivity}
The map $\cB(A,F) \to \Hom( \Gal_F,\, \GL(\rH^1(A,\bZ_\ell)))$ is injective.
\end{lemma}

\begin{proof}
Let $(\cA,\alpha)$ and $(\cB,\beta)$ be models over $F$ of $A$. The isomorphism $\beta^{-1} \alpha$ from $\cA_\bC$  to $\cB_\bC$ is defined over $\bar F$. Now assume that the models define the same  action on $\rH^1(A,\bZ_\ell)$. Then for every $\sigma \in \Gal_F$ the isomorphisms $\cA_{\bar F} \to \cB_{\bar F}$ given by $\beta^{-1} \alpha$ and by $\sigma^{-1}\beta^{-1} \alpha \sigma$ induce the same isomorphism  $\rH^1_\et( \cB_{\bar F}, \bZ_\ell) \to\rH^1_\et( \cA_{\bar F}, \bZ_\ell)  $, and hence coincide. It follows that $\beta^{-1} \alpha$ is defined over $F$, and that $(\cA,\alpha)$ and $(\cB,\beta)$ give the same element of $\cB(A,F)$.
\end{proof}

If $A$ has complex multiplication, then the main theorem of CM implies that $A$ can be defined over a number field, and suggests that it should be possible to describe the image of the injective map of
Lemma \ref{lemma:injectivity}. One of the main results of this note is such a description, under the condition that $F$ contains the reflex field of $A$, and using $\hat\bZ$-coefficients instead of $\bZ_\ell$-coefficients. We now give the precise statement.

Assume $A$ has complex multiplication, and let $T$ be the Mumford-Tate group of $A$. This is an algebraic torus over $\bQ$. Class field theory provides a canonical homomorphism
\[
	\rec\colon \Gal_E \to  T(\bA_f)/T(\bQ)
\]
where $E\subset \bC$ is the reflex field of $A$ (see \S~\ref{sec:reflex-field} for more details). 
Denote by $C$ the centralizer of $T$ in $\GL(\rH^1(A,\bQ))$. This is an algebraic group whose $\bQ$-points can be identified with the group of units in the algebra $\bQ \otimes \End A$. The map
\[
	C(\bQ)T(\bA_f) \to T(\bA_f)/T(\bQ),\, ct\mapsto t
\]
is well-defined since $C(\bQ) \cap T(\bA_f) = T(\bQ)$.

\begin{theorem}\label{thm:AV}
Let $F\subset \bC$ be a number field containing $E$. Then the image of the injective map
\[
	\cB(A,F) \to \Hom( \Gal_F,\, \GL(\rH^1(A,\hat\bZ)))
\]
consists precisely of those $\rho\colon \Gal_F \to \GL(\rH^1(A,\hat\bZ))$ satisfying
\begin{enumerate}
\item the image of $\rho$ is contained in $C(\bQ) T(\bA_f) \subset \GL(\rH^1(A,\bA_f))$,
\item the diagram
\[
\begin{tikzcd}
\Gal_F \arrow[hook]{r} \arrow{d}{\rho} & \Gal_E \arrow{d}{\rec} \\
C(\bQ)T(\bA_f) \arrow{r}  &  T(\bA_f)/T(\bQ)
\end{tikzcd}
\]
commutes,
\item there exists a polarization $\Psi\colon \rH^1(A,\bQ) \times \rH^1(A,\bQ) \to \bQ(-1)$
such that the image of $\rho$ lands in $\GSp(\rH^1(A,\bA_f),\Psi)$.
\end{enumerate}
\end{theorem}

\begin{remark}
The image of a Galois representation $\rho$ satisfying (1)--(3) need not be abelian, and indeed it is easy to construct a CM abelian variety over an extension $F$ of its reflex field such that the action of Galois on its torsion points is not abelian.

For example, let $A$ over $F=\bQ(i)$ be the square of an elliptic curve $E$ with $\bZ[i]$-action, let $\gamma \colon \Gal_F \to \GL_2(\bZ[i])$ be a homomorphism with non-abelian image, interpret $\gamma$ as a $1$-cocycle and let $A^\gamma$ be the twist of $A$ by $\gamma$ acting via $ \GL_2(\bZ[i]) = \Aut(A)$. Then the representation of $\Gal_F$ on $\rH^1(A^\gamma_{\bar F},\bQ_\ell)$ is isomorphic to the non-abelian representation $V \otimes_{\bQ(i)} \rH^1(E_{\bar F},\bQ_\ell)$ with $V$ the two-dimensional representation over $\bQ(i)$ given by $\gamma$.  Note that the Mumford-Tate group of $A$ is not a maximal torus in $\GSp_4$. 
\end{remark}

\subsection{K3 surfaces}\label{subsec:intro-K3}
Our second theorem is an analogue of Theorem \ref{thm:AV} for K3 surfaces with complex multiplication. The reader who is only interested in abelian varieties and Shimura stacks, can skip ahead to \S~ \ref{subsec:intro-Sh}.

Let $X$ be a K3 surface over $\bC$. Denote by  $V_X$ its ($\bQ$-)transcendental lattice of $X$, which is defined as  the orthogonal complement of $(\Pic X)\otimes \bQ$ inside $\rH^2(X,\bQ(1))$. The endomorphism ring $E := \End V_X$ (in the category of $\bQ$-Hodge structures) is a field \cite{Zarhin83}, and we say that $X$ has \emph{complex multiplication} (by $E$) if $V_X$ is one-dimensional as an $E$-vector space.

Assume that $X$ is a complex K3 surface with complex multiplication by $E$. Then the Mumford-Tate group of the $\bQ$-Hodge structure $V_X$ is the algebraic torus $T$ with
\[
	T(\bQ) := \{ z \in E^\times \mid z\bar z = 1 \}.
\]
As above, class field theory gives a canonical reciprocity map
\[
	\rec\colon \Gal_E \to T(\bA_f)/T(\bQ),
\]
see \S~\ref{sec:reflex-field} for more details. 
 For a subfield $F$ of $\bC$, define $\cB(X,F)$ as the set of isomorphism classes of pairs $(\cX,\alpha)$ with $\cX$ a K3 surface over $F$ and $\alpha$ an isomorphism $\cX_{\bC}\to X$. As for abelian varieties, every such pair $(\cX,\alpha)$ defines a Galois representation $\rho\colon \Gal_F \to \Orth(\rH^2(X,\hat\bZ(1)))$.

\begin{theorem}\label{thm:K3}
Let $F\subset \bC$ be a finite extension of $E$. Then the map
\[
	\cB(X,F) \to \Hom(\Gal_F, \Orth(\rH^2(X,\hat\bZ(1)) )
\]
is injective, and its image consists precisely of those
\[
	\rho\colon \Gal_F\to\Orth(\rH^2(X,\hat\bZ(1)))
\]
satisfying
\begin{enumerate}
\item[(1)] the action of $\Gal_F$ on $V_{X,\bA_f}$ factors over $T(\bA_f) \subset \Orth(V_{X,\bA_f})$ and the diagram
\[
\begin{tikzcd}
\Gal_F \arrow[hook]{r} \arrow{d}{\rho_{|V_{X,\bA_f}}} & \Gal_E \arrow{d}{\rec} \\
T(\bA_f) \arrow{r}  &  T(\bA_f)/T(\bQ)
\end{tikzcd}
\]
commutes,
\item[(2)] the action of $\Gal_F$ preserves $\Pic X \subset \rH^2(X,\hat\bZ(1))$,
\item[(3)] the action of $\Gal_F$ on $(\Pic X)\otimes \bR$ preserves the ample cone
 $K_X$.
\end{enumerate}
\end{theorem}

\subsection{CM points on Shimura stacks}\label{subsec:intro-Sh} Theorems \ref{thm:AV} and \ref{thm:K3} will be deduced from a general statement about CM points on Shimura stacks.

Let $(G,X)$ be a Shimura datum, $\cK$ a profinite group and $\cK \to G(\bA_f)$ a continuous homomorphism with open image and finite kernel. Then one can form the complex holomorphic quotient stack
\[
	\Sh_\cK^\an[G,X] :=  \Big[ \, G(\bQ) \,\backslash\, \big( X \times G(\bA_f)/\cK \big)\, \Big].
\]
See \S~\ref{sec:stack} for more details. The groupoid of complex points on $\Sh_\cK^\an[G,X]$ is equivalent to the following groupoid:
\begin{enumerate}
\item objects: pairs $(h,g)$ in $X\times G(\bA_f)$
\item morphisms from $(h_0,g_0)$ to $(h_1,g_1)$: pairs $(\gamma,k) \in G(\bQ) \times \cK$ satisfying
$\gamma h_0 = h_1$ and $\gamma g_0 k = g_1$.
\end{enumerate}
Let $Z$ be the center of $G$. In this paper all Shimura data will satisfy
\begin{equation}
\text{ $Z(\bQ)$ is discrete in $Z(\bA_f)$.} \tag{$\star$}
\end{equation} 
This condition is denoted SV5 in \cite{Milne04}, and is implied by condition (2.1.1.5) in \cite{Deligne79}. Under this condition, the stabilizers in $\Sh_{\cK}^\an[G,X]$ are finite, and this holomorphic stack  can be canonically given the structure of a smooth Deligne-Mumford stack $\Sh_\cK[G,X]$ over the reflex field of $(G,X)$, see \S~\ref{sec:stack} for more details.

\begin{remark}
We have slightly generalized the usual set-up by allowing the map $\cK \to G(\bA_f)$ not to be injective. This will be useful in the proof of Theorem \ref{thm:K3}. In fact, as we will see, polarized K3 surfaces of degree $2$ are parametrized by a Shimura stack for which the map $\cK\to G(\bA_f)$ is not injective.
\end{remark}

Let $x =(h,g)$ be an object in $\Sh_\cK^\an[G,X]$, and let $F$ be a subfield of $\bC$ containing the reflex field of $(G,X)$. By analogy with the above, we define $\cB(x,F)$ to be the set of isomorphism classes of pairs $(\xi,\alpha)$ consisting of an object $\xi$ in $\Sh_\cK[G,X](F)$ and an isomorphism $\alpha \colon \xi_\bC \isomto x$. If $\cK$ is small enough (e.g.~a neat subgroup of $G(\bA_f)$) then $\Sh_\cK[G,X]$ is representable and hence $\cB(x,F)$ consists of at most one element.

Assume that $x=(h,g)$ is a \emph{special point}, so that the Mumford-Tate group of $h$ is a torus $T\subset G$. Let $C\subset G$ be the centralizer of $T$ and let $E$ be the reflex field of $h$. See \S~\ref{sec:some-groups}  and \S~\ref{sec:reflex-field} for a more detailed description of $T$, $C$, and $E$. 

\begin{theorem}\label{thm:Sh} Let $F \subset \bC$ be a finite extension of $E$.
Then there is a natural injective map
\[
	\cB(x,F) \to \Hom(\Gal_F,\, g\cK g^{-1})
\]
whose image consists of those $\rho \colon \Gal_F \to g\cK g^{-1}$ such that
\begin{enumerate}
\item the image of $\rho$ in $G(\bA_f)$ is contained in $C(\bQ) T(\bA_f)$,
\item the diagram
\[
\begin{tikzcd}
\Gal_F \arrow[hook]{r} \arrow{d}{\rho} & \Gal_E \arrow{d}{\rec} \\
C(\bQ)T(\bA_f) \arrow{r}  & T(\bA_f)/T(\bQ)
\end{tikzcd}
\]
commutes.
\end{enumerate}
\end{theorem}

As in Theorem \ref{thm:AV}, the lower horizontal map in (2) given by $ct \mapsto t$.

\begin{remark}The notation $g\cK g^{-1}$ should be interpreted as the fiber product
\[
\begin{tikzcd}
g\cK g^{-1} \arrow{r} \arrow{d} & G(\bA_f) \arrow{d}{g^{-1}(-)g} \\
\cK \arrow{r} & G(\bA_f)
\end{tikzcd}
\]
Of course, if $\cK \to G(\bA_f)$ is injective, then $g\cK g^{-1}$ is just the conjugated subgroup.
\end{remark}

\begin{remark}\label{rmk:lift}
Consider the fiber product
\[
	g\cK g^{-1} \times_{G(\bA_f)} C(\bQ)T(\bA_f)
\]
(if $\cK \to G(\bA_f)$ is injective, then  this fiber product is  the intersection $ g\cK g^{-1} \cap C(\bQ)T(\bA_f)$ in $G(\bA_f)$), and consider the  homomorphism
\[
	 g\cK g^{-1} \times_{G(\bA_f)} C(\bQ)T(\bA_f) \overset\delta\longto T(\bA_f)/T(\bQ),
\]
given by $(gkg^{-1},ct) \mapsto t$.  The set of $\rho$ satisfying the two conditions in the theorem is the set of
\[
	\rho \colon \Gal_F \to  g\cK g^{-1} \times_{G(\bA_f)} C(\bQ)T(\bA_f)
\]
\emph{lifting} the reciprocity map $\Gal_F \to T(\bA_f)/T(\bQ)$ along $\delta$. As such, the standard observations about lifting group homomorphisms apply:
\begin{enumerate}
\item If $\rho_0$ and $\rho_1$ are two lifts, then $\rho_1/\rho_0 \colon \Gal_F \to \ker \delta$ is a $1$-cocycle (corresponding to the fact that any two $F$-points $\xi$ under a given $\bC$-point $x$ are twists of each other under $\Aut x = \ker \delta$, see Corollary \ref{cor:automorphism-group}).
\item A necessary condition for $\rho$ to exist is that the image of $\Gal_F$ in $\coker \delta$ vanishes (this implies, for example, the well-known fact that the field of definition of a CM elliptic curve $A$ contains the Hilbert class field of the quadratic imaginary field $\bQ\otimes \End A$).
\item Even if the above condition is satisfied, a lift $\rho$ need not exist. For example, if $\ker \delta$ is abelian, then there is an obstruction in $\rH^2(\Gal_F, \ker \delta)$ against the existence of a lift.
\end{enumerate}
\end{remark}

\begin{remark}If the Mumford-Tate group $T$ is a \emph{maximal} torus in $G$, then $T=C$ and every $\rho$ in the theorem has abelian image.
\end{remark}

\subsection{About the proofs}

As one may expect, Theorem \ref{thm:Sh} is a rather formal consequence of the reciprocity law describing the action of Galois on CM points on Shimura varieties. If $\cK$ is a neat subgroup of $G(\bA_f)$ then $\Sh_{\cK}[G,X]$ is representable (so that there is at most one element in $\cB(x,F)$) and the map
$\delta$ of Remark \ref{rmk:lift} is injective (so that there is at most one $\rho$ making the square in Theorem \ref{thm:Sh} commute).  In this case, the theorem boils down to a tautology: it describes the action of Galois on CM points on Shimura varieties that is used to \emph{define} the canonical model of $\Sh_\cK(G,X)$. The general case is essentially obtained by descent along the $\cK/\cK_0$-torsor $\Sh_{\cK_0}(G,X) \to \Sh_{\cK}[G,X]$ for a neat normal $\cK_0 \subset \cK$. 

Theorems \ref{thm:AV} and \ref{thm:K3} are obtained from Theorem \ref{thm:Sh} using the fact that the moduli stacks of polarized abelian varieties (resp.~of polarized K3 surfaces) are Shimura stacks (resp.~open substacks of Shimura stacks). Condition (3) in both Theorem \ref{thm:AV} and Theorem \ref{thm:K3} arrange passage from the `unpolarized' setting of representations with
image in groups of type $\GL(2g)$ or $\Orth(3,19)$, to the polarized Shimura setting, with groups of type
$\GSp(2g)$ or $\SO(2,19)$.

In order to show the result on K3 surfaces, we use that the stack of polarized K3 surfaces over $\bQ$ forms an open substack of a Shimura stack. The precise result seems to be missing from the literature.
As this may be of independent interest, we give a detailed description in \S~\ref{subsec:K3-Shimura} and \S~\ref{subsec:level-structures}. Our treatment is heavily inspired by Andr\'e \cite{Andre96a} and Madapusi Pera \cite{MadapusiPera15}, whose work make it clear that one should not only consider the intersection pairing on $\rH^2$ of a K3 surface, but also the trivializations of its determinant.

\section{The stack $\Sh_{\cK}[G,X]$} \label{sec:stack}

\subsection{Definition}
Let $(G,X)$ be a Shimura datum satisfying ($\star$). Let $\cK$ be a profinite group and let $\cK\to G(\bA_f)$ be a continuous homomorphism with open image and finite kernel. The holomorphic quotient stack
\[
	 \Sh_\cK^\an[G,X] :=  \Big[ \, G(\bQ) \,\backslash\, \big( X \times G(\bA_f)/\cK \big)\, \Big]
\]
is a finite disjoint union of stacks of the form $[\Gamma \backslash X]$, with $\Gamma$ acting via a homomorphism $\Gamma \to G(\bQ)$ with finite kernel and image an arithmetic subgroup.

The isomorphism classes in $\Sh_\cK^\an[G,X]$ are the double cosets in
\[
	 \Sh_\cK^\an(G,X) :=  G(\bQ) \,\backslash\, \big( X \times G(\bA_f)/\cK \big),
\]
 and the stabilizer $\Aut x$ of a point $x=(h,g)$ is given by the cartesian square
\[
\begin{tikzcd}
\Aut x \arrow{r} \arrow{d} & g\cK g^{-1} \arrow{d} \\
G(\bQ)_h \arrow{r} & G(\bA_f) 
\end{tikzcd}
\]
where $G(\bQ)_h$ denotes the stabilizer of $h\in X$ in $G(\bQ)$.
If $\cK$ is a neat subgroup of $G(\bA_f)$, then the stabilizers are trivial and $\Sh_\cK^\an[G,X]$ is represented by $\Sh_\cK^\an(G,X)$, the analytification of the (smooth, quasi-projective) Shimura variety $\Sh_\cK(G,X)$. 

Choose an open normal subgroup $\cK_0\subset \cK$ such that $\cK_0\to G(\bA_f)$ is injective with neat image. Then the group $\cK/\cK_0$ acts on the Shimura variety $\Sh_{\cK_0}(G,X)$, and this action is defined over the reflex field $E(G,X)$. We define the \emph{Shimura stack} $\Sh_{\cK}[G,X]$ as the quotient stack
\[
	\Sh_\cK[G,X] :=  \big[ \Sh_{\cK_0}(G,X) / (\cK/\cK_0) \, \big],
\]
over $E$. This is independent of the choice of $\cK_0$, and we have $\Sh_\cK[G,X](\bC) = \Sh_\cK^\an[G,X]$. The stack  $\Sh_{\cK}[G,X]$ is a smooth Deligne-Mumford stack over $E$.

\begin{example}[Elliptic curves]\label{exa:elliptic-curves}
 The standard example is the stack of elliptic curves. Take $G=\GL_2$, $X=\cH^\pm$ the double upper half plane, and $\cK = \GL_2(\hat\bZ) \subset G(\bA_f)$. The reflex field of $(G,Q)$ is $\bQ$ and $\Sh_{\cK}[G,X]$ is canonically isomorphic to the Deligne-Mumford stack $\Ell_\bQ$ of elliptic curves over $\bQ$. The analytification is the quotient stack
\[
	\Sh_{\cK}[G,X]^\an = \big[ \SL_2(\bZ) \backslash \cH^+ \big].
\]
Every point has a non-trivial finite stabilizer (the automorphism group of the corresponding elliptic curve). The $j$-invariant induces a commutative diagram of stacks over $\Spec \bQ$ 
\[
\begin{tikzcd}
	\Sh_{\cK}[G,X] \arrow{d} \arrow{r}{\cong} & \Ell_\bQ  \arrow{d}{j} \\
	 \Sh_{\cK}(G,X) \arrow{r}{\cong} & \bA^1_\bQ
\end{tikzcd}
\]
See \S~\ref{subsec:K3-Shimura} for a natural example where $\cK\to G(\bA_f)$ is not injective.
\end{example}

\subsection{Variations of Hodge structures on Shimura stacks}\label{subsec:VHS}

Let $V$ be a finite dimensional $\bQ$-vector space equipped with an algebraic left action of $G$ and a continuous right action of $\cK$ (necessarily with finite image). Assume that these actions commute. Then  the double quotient 
\[	
 \Big[ \, G(\bQ) \,\backslash\, \big( X \times V \times G(\bA_f) \big) / \cK \, \Big] \longto
 \Sh_\cK[G,X]
\]
defines a variation of $\bQ$-Hodge structures $\cV_\bQ$ on $\Sh_{\cK}[G,X]$. This is slightly more general than the usual construction of variations of Hodge structures on Shimura varieties, where the action of $\cK$ on $V$ is trivial. The added flexibility will be of use in our treatment of K3 surfaces.

Now let $V_{\hat\bZ} \subset V\otimes\bA_f$ be a $\hat\bZ$-lattice and assume
\[
	\iota(k)^{-1} v k \in V_{\hat\bZ} \text{ for all }v \in V_{\hat\bZ},
\]
where  $\iota$ is the map $\cK \to G(\bA_f)$. Then  the double quotient 
\[
	\Big[ \, G(\bQ) \,\backslash\, 
	\big\{ (h,v,g) \in X \times V \times G(\bA_f) \mid g^{-1}v \in V_{\hat\bZ} \big\} / \cK \, \Big]
\] 
defines a variation of $\bZ$-Hodge structures $\cV_{\bZ} \subset \cV_\bQ$. 

\section{Proof of Theorem \ref{thm:Sh} (Shimura stacks)}

\subsection{Some groups associated with a CM point}\label{sec:some-groups}
Let $(G,X)$ be a Shimura datum satisfying ($\star$), let $\cK$ a profinite group, and let $\cK \to G(\bA_f)$ be a continuous homomorphism with finite kernel and open image. Let $x=(h,g)$ be an object in $\Sh_{\cK}[G,X](\bC)$.  We say that $x$ is a \emph{special point} or \emph{CM point} if its Mumford-Tate group is an algebraic torus.

Let $x=(h,g)\in \Sh_{\cK}(G,X)(\bC)$ be a special point with Mumford-Tate group $T$. Let $C\subset G$ be the centralizer of $T$. 

\begin{lemma}\label{lemma:stab-is-centralizer}
The stabilizer of $h\in X$ in $G(\bQ$) is $C(\bQ)$.
\end{lemma}

\begin{proof}
By assumption, $T\subset G$ is the smallest $\bQ$-torus such that $h\colon \bS \to G_\bR$ factors over $T_\bR$. Clearly $C(\bQ)$ stabilizes $h$. Conversely, assume that $\gamma \in G(\bQ)$ stabilizes $h$. Then it stabilizes $T$, and induces an automorphism $\sigma\in \Aut T$, defined over $\bQ$. It now suffices to show that $\sigma$ is the identity.

Let $T_0\subset T$ be the connected component of identity of the subgroup fixed by $\sigma$. This is a proper sub-torus of $T$ and by construction the morphism $h$ factors over $T_0(\bR)$. It follows that $T_0=T$, and hence $\sigma=\id_T$.
\end{proof}

\begin{corollary}\label{cor:automorphism-group}
The stabilizer of $(h,g) \in \Sh_{\cK}^\an[G,X]$ is 
 $g\cK g^{-1} \times_{G(\bA_f)} C(\bQ)$.\qed
\end{corollary}

The product $C(\bQ)T(\bA_f)$ is a subgroup of $G(\bA_f)$, and we have a well-defined group homomorphism
\[
	g\cK g^{-1} \times_{G(\bA_f)} C(\bQ) T(\bA_f) \overset{\delta}\longto T(\bQ)\backslash T(\bA_f)
\]
given by $(gkg^{-1},ct) \mapsto t$. Both kernel and cokernel of $\delta$ are finite. By Corollary \ref{cor:automorphism-group} the kernel is the stabilizer of $(h,g)$ in the stack $\Sh_{\cK}^\an[G,X]$.

\subsection{Reflex field and reciprocity map}\label{sec:reflex-field}
Let $T$ be an algebraic torus over $\bQ$ and let $h\colon \bS \to T_\bR$ be a morphism of tori over $\bR$. With $h$ is associated a number field
$E \subset \bC$ called the \emph{reflex field} of $h$, and a map of  $\bQ$-tori
\[
	t\colon \Res_{E/\bQ} \bG_{m,E} \longto T.
\]
The subfield $E \subset \bC$ and the map $t$ can be characterized as follows. There is a unique isomorphism of
$\bC$-tori
\[
	\bS_\bC \longisomto \bG_{m,\bC} \times \bG_{m,\bC}
\]
such that $z \in \bC^\times = \bS(\bR) \subset \bS(\bC)$ is mapped to $(z,\bar z)$. The map $t$ has the property that the base change of the composite map 
\[
	\bG_{m} \longto \Res_{E/\bQ} \bG_{m,E} \overset{t}{\longto} T
\]
from $\bQ$ to $\bC$ coincides with the map of $\bC$-tori
\[
	\bG_{m,\bC} \overset{(z,1)}{\longto} \bS_\bC \overset{h}{\longto} T_\bC.
\]
The reflex field $E$ is the unique minimal field $E$ for which such $t$ exists.
The group $T(\bA_f)/\widebar{T(\bQ)}$ is profinite, so  the composition
\[
	\bA_E^\times/E^\times \overset{t}{\longto} T(\bA)/T(\bQ) 
	\longto T(\bA_{f})/\widebar{T(\bQ)}
\]
factors over the profinite completion of $ \bA_E^\times/E^\times $, which class field theory identifies with $\Gal_E^\ab$. We obtain a homomorphism
\[
	\rec\colon \Gal_E \to T(\bA_f)/\widebar{T(\bQ)}
\]
which (by a slight abuse of terminology) we will call the reciprocity map. In this paper, $T(\bQ)$ will always be discrete in $T(\bA_{f})$, so that $\rec$ will be a homomorphism $\Gal_E \to T(\bA_f)/T(\bQ)$.

\subsection{Recap on canonical models}

We summarize some of the main results on existence of canonical models of Shimura varieties \cite{Shimura70,Deligne79,Borovoi82,Milne83}.

With a Shimura datum $(G,X)$ is associated a canonical number field $E(G,X) \subset \bC$ called the reflex field. For every compact open subgroup $\cK\subset G(\bA_f)$ a model of $\Sh_\cK(G,X)$ over $E(G,X)$. 
These were constructed for a large class of $(G,X)$ in \cite[Cor.~2.7.21]{Deligne79}, and later for arbitrary Shimura data in \cite{Borovoi82,Milne83}.

These models are functorial in the following sense:  for every morphism of Shimura data $\phi\colon(G_1,X_1) \to (G_2,X_2)$ and for every pair of compact open subgroups $\cK_i\subset G_i(\bA_f)$ with $\phi(\cK_1) \subset \cK_2$
we have $E(G_2,X_2) \subset E(G_1,X_1)$ and the induced map
\[
	\Sh_{\cK_1}(G_1,X_1) \to \Sh_{\cK_2}(G_2,X_2)
\]
is defined over $E(G_1,X_1)$. The collection of all these models satisfies (and is characterized by) the following property (which is part of the definition of a canonical model, see \cite[\S~2.2]{Deligne79}).

\begin{theorem}[Reciprocity law]\label{thm:reciprocity}
Let $(T,\{h\})$ be a Shimura datum with $T$ a torus. Let $E$ be its reflex field. Let $\cK \subset T(\bA_f)$ be a compact open subgroup with $\cK \cap T(\bQ) = \{1\}$.
Then $\Sh_\cK(T,\{h\})$ is a finite \'etale scheme over $E$ and the action of $\Gal_E$ on
its $\bC$ points
\[
	\Sh_\cK^\an(T,\{h\}) =  T(\bQ) \backslash T(\bA_f)/\cK
\]
is obtained by composing the multiplication action of $\widebar{T(\bQ)} \backslash T(\bA_f)$ with
the  reciprocity homomorphism $\rec\colon\Gal_E \to \widebar{T(\bQ)} \backslash T(\bA_f)$.\qed
\end{theorem}

\subsection{A lemma on finite \'etale quotient stacks}

Let $F$ be a field with separable closure $\bar F$. Let $f\colon \Gamma_1 \to \Gamma_2$ be a homomorphism of finite groups, and let $\rho_2\colon \Gal_F \to \Gamma_2$ be a continuous homomorphism. Then $\rho_2$ determines a finite \'etale $F$-scheme $X_2$ with an identification $X_2(\bar F) = \Gamma_2$. 

If $\Gamma_2$ is moreover commutative, then $\Gamma_1$ acts on $X_2$, and we have a quotient stack $\cX := [ X_2/\Gamma_1 ]$. The unit element $1 \in \Gamma_2 = X_2(\bar F)$ determines an object
$x$ in $\cX(\bar F)$. Let $\cB(x,F)$ be the set of isomorphism classes of pairs $(\xi,\alpha)$ with $\xi$ an object of $\cX(F)$ and $\alpha \colon \xi_{\bar F} \to x$.

\begin{lemma} \label{lemma:finetquot}
There is a natural bijection between $\cB(x,F)$ and the set of $\rho_1\colon \Gal_F \to \Gamma_1$ making the diagram
\[
\begin{tikzcd}
& \Gal_F \arrow[dashed,swap]{dl}{\rho_1} \arrow{d}{\rho_2} \\
\Gamma_1 \arrow{r}{f} & \Gamma_2
\end{tikzcd}
\]
commute.
\end{lemma}

\begin{proof}
A morphism $\rho_1\colon \Gal_F \to \Gamma_1$ determines a $\Gamma_1$ torsor $X_1$ over $F$ with an identification $X_1(\bar F) = \Gamma_1$. If $\rho_1$ satisfies $f\rho_1=\rho_2$ then the map
$X_{1,\bar F} \to X_{2,\bar F}$ given by $f\colon \Gamma_1 \to \Gamma_2$ is invariant under conjugation by $\Gal_F$, and hence descends to a map $\varphi$ defined over $F$. The pair $(X_1,\varphi)$ defines an $F$-point of $[X_2/\Gamma_1]$ and by construction it comes with an isomorphism $ (X_1,\varphi)_{\bar F} \cong (\Gamma_1, f)$. It is clear that this construction defines a bijection.
\end{proof}

\subsection{Proof of the theorem on CM points on Shimura stacks}

\begin{proof}[Proof of Theorem \ref{thm:Sh}]
Choose a normal open subgroup $\cK_0\subset \cK$ such that $\cK_0 \to G(\bA_f)$ is injective with neat image. Consider the groups
\[
	\cU := g\cK g^{-1} \times_{G(\bA_f)} C(\bQ) T(\bA_f)
\]
and
\[
	\cU_0 := g\cK_0 g^{-1} \times_{G(\bA_f)} C(\bQ) T(\bA_f) =  g\cK_0 g^{-1}  \cap C(\bQ) T(\bA_f)
\]
Shrinking $\cK_0$ if necessary, we may assume that $\cU_0$ is contained in $T(\bA_f)$, and that it maps injectively to $T(\bQ)\backslash T(\bA_f)$, 

The zero-dimensional Shimura variety $\Sh_{\cU_0}(T,\{h\})$ is defined over $E$, and carries an action of $\cU/\cU_0$, also defined over $E$. Denote the quotient stack by $\Sh_{\cU}[T,\{h\}]$. We have
\[	
	\Sh_{\cU}^\an[T,\{h\}] = \Big[ \big( T(\bQ) \backslash T(\bA_f) \big) / \, \cU \Big].
\]
(This is not strictly speaking speaking a Shimura stack, since the map $\cU \to T(\bQ)\backslash T(\bA_f)$ need not lift to a map $\cU\to T(\bA_f)$.) The  map
\[
	 \Sh_{\cU_0}(T,\{h\}) \to \Sh_{\cK_0}(G,X), \, [h,t] \mapsto [h,tg]  
\]
is defined over $E$ and is equivariant with respect to the homomorphism
\[
	\cU/\cU_0 \to \cK/\cK_0,\, u \mapsto g^{-1}ug
\]	
of groups acting. We obtain an induced map
\[
	\iota\colon\Sh_\cU[T,\{h\}] \to \Sh_{\cK}[G,X]
\]
of Deligne-Mumford stacks over $E$. We claim that it is a closed immersion. Indeed, it suffices to verify that the map is fully faithful on $\bC$-points. Explicitly, the map $\iota(\bC)$ is given by
\[
	\Big[ T(\bQ) \backslash T(\bA_f) / \cU \Big] \to
	 \Big[ \, G(\bQ) \,\backslash\, X \times G(\bA_f)/\cK \, \Big],\,
	 t \mapsto (h,tg).
\]
If $t_1,t_2 \in T(\bA_f)$ represent objects $x_1,x_2$ in the source, then we have
\[
	\Hom(x_1,x_2) = \{ (\tau,u) \in T(\bQ) \times \cU \mid \tau t_1u=t_2 \}
\]
and, using Lemma \ref{lemma:stab-is-centralizer} we find
\[
	\Hom(\iota x_1,\iota x_2) = \{ (\gamma,k) \in C(\bQ)\times \cK   \mid \gamma t_1gk = t_2 g \}.
\]
The map $\iota$ between these Hom-sets is given by $(\tau,u) \mapsto (\tau, g^{-1} u g)$, and one verifies directly that this is a bijection.

We have $x=\iota y$ with $y=(h,1)$, and since $\iota$ is a closed immersion, it defines a bijection $\cB(y,F) \to \cB(x,F)$. By Theorem \ref{thm:reciprocity} and Lemma \ref{lemma:finetquot}, the map
\[
	\cB(y,F) \to \Hom(\Gal_F,\,\cU/\cU_0)
\]
is injective, with as image precisely those $\bar\rho$ making the diagram
\[
\begin{tikzcd}
\Gal_F \arrow{r} \arrow{d}{\bar\rho} & \Gal_E \arrow{d}{\rec} \\
\cU/\cU_0 \arrow{r} & T(\bQ) \backslash T(\bA_f)/\cU_0
\end{tikzcd}
\]
commute. Since $\cU_0 \to T(\bQ) \backslash T(\bA_f)$ is injective, the square of group homomorphisms
\[
\begin{tikzcd}
 \cU \arrow{r} \arrow{d} &T(\bQ) \backslash T(\bA_f) \arrow{d} \\
 \cU/\cU_0 \arrow{r} & T(\bQ) \backslash T(\bA_f)/\cU_0 
\end{tikzcd}
\]
is cartesian. This shows that there is a bijection between the set of morphisms $\rho\colon \Gal_F\to \cU$ making the square in Theorem \ref{thm:Sh} commute, and the set of morphisms $\bar\rho\colon \Gal_F\to \cU/\cU_0$ as above, which finishes the proof of the theorem.
\end{proof}

\section{Proof of Theorem \ref{thm:AV} (Abelian varieties)}

We use the notation of \S~\ref{subsec:intro-AV}, so $A/\bC$ is an abelian variety with complex multiplication, 
$T$ is the Mumford-Tate group of $A$ and $C$  is the centralizer of $T$ in $\GL(\rH^1(A,\bQ))$.

\begin{proof}[Proof of Theorem \ref{thm:AV}]
If $\Psi\colon \wedge^2 \rH^1(A,\bZ) \to \bZ(1)$ is a polarization, then we write $\GSp_{\Psi}$ for the reductive group
$\GSp(\rH^1(A,\bQ), \Psi)$. Consider the Siegel Shimura datum $(\GSp_\Psi, \cH_g)$ and the compact open subgroup
\[
	\cK := \GSp_\Psi(\bA_f) \cap \GL(\rH^1(A,\hat\bZ)).
\]
The Shimura stack $\Sh_\cK[\GSp_\Psi, \cH_g]$ is the moduli stack of abelian varieties with a polarization `of type $\Psi$'. More specifically, for a field $F$ of characteristic $0$, the groupoid $\Sh_\cK[\GSp_\Psi, \cH_g](F)$ is equivalent to the groupoid of pairs $(\cA,\lambda)$ with
\begin{enumerate}
\item $\cA$ an abelian variety over $F$
\item $\lambda\colon \cA \to \cA^\vee$ a polarization
\end{enumerate}
for which there exists an isomorphism of $\hat\bZ$-modules $\rH^1_\et(\cA_{\bar F},\hat\bZ) \isomto \rH^1(A,\hat\bZ)$ respecting the pairings induced by $\lambda$ and $\Psi$.  The pair $(A,\Psi)$ defines a complex point $x=(h,1)$ on $\Sh_\cK[\GSp_\Psi, \cH_g]$, and if $(\cA,\lambda,\alpha)$ is a model of $(A,\Psi)$ over $F$ corresponding to a point $\xi$ under $x$, then the Galois representation
$\rho_{\cA} \colon \Gal_F \to\GL(\rH^1(A,\hat\bZ))$ (coming from the action on \'etale cohomology) and $\rho_{\xi } \colon \Gal_F \to \cK$ (as produced by Theorem \ref{thm:Sh}) coincide. 

Now if $(\cA,\alpha)$ is a model of $A$ over $F$, then there exists a polarization $\Psi$ on $A$ that descends to  $\cA/F$, and hence $\cA$ defines an $F$-point on $\Sh_\cK[\GSp_\Psi, \cH_g]$. Theorem \ref{thm:Sh} then implies that the Galois representation $\rho$ associated with $(\cA,\alpha)$ satisfies conditions (1) and (2) in Theorem \ref{thm:AV}, and it clearly also satisfies (3).
 
Conversely, if $\rho$ is as in the theorem, then choose a polarization $\Psi \colon \wedge^2 \rH^1(A,\bZ) \to \bZ(-1)$ 
as in (3). Then the pair $(A,\Psi)$ defines a complex point on the Shimura stack $\Sh_\cK[\GSp_\Psi, \cH_g]$, and by Theorem \ref{thm:Sh}, the representation $\rho$ corresponds to an $F$-point `under' $(A,\Psi)$, and hence to a model $\cA$ of $F$.
\end{proof}

\section{Proof of Theorem \ref{thm:K3} (K3 surfaces)}

\subsection{Moduli of polarized K3 surfaces as a Shimura stack}\label{subsec:K3-Shimura}

Let $\Lambda$ be the K3 lattice (the unique even self-dual lattice of signature $(3,19)$, see \cite[Prop.~3.5]{HuybrechtsK3}), let $d$ be a positive integer, and fix a primitive $\lambda \in \Lambda$ with $\lambda^2=2d$ (such $\lambda$ exists and is unique up to isometries of $\Lambda$).  Let $\Lambda_{2d}$ be the orthogonal complement of $\lambda$. We thus have a decomposition $\Lambda \otimes \bQ = (\Lambda_{2d} \otimes \bQ) \oplus \bQ\lambda$.

\begin{definition}A \emph{polarization of degree $2d$} on a K3 surface $X\to S$ is a section $\cL \in \Pic_{X/S}(S)$ such that $\cL_s \in \Pic_{X_s/\kappa(s)}$ is primitive ample and satisfies $\cL_s\cdot \cL_s=2d$ for all $s\in S$.
\end{definition}

If $(X,\cL)$ is a polarized K3 surface of degree $2d$ over $\bC$, then there exists an isometry $\rH^2(X,\bZ(1)) \isomto \Lambda$ mapping $c_1(\cL)$ to $\lambda$.

Let $\cM_{2d}$ denote the stack of polarized K3 surfaces of degree $2d$ over $\bQ$.  By the arguments of \cite{Rizov06}, this is a Deligne-Mumford stack. We will show that it is an open substack of a Shimura stack.

\medskip

Consider the profinite group
\[
	\cK := \big\{ g \in \Orth(\hat\bZ \otimes \Lambda) \mid
	\det g \in \{\pm 1\},\, g\lambda = \lambda \big\}.
\]
The condition  $\det g  \in \{\pm 1\}$ means that we ask that {either} for every place $v$ we have $\det g_v = 1$, {or} that for every prime $v$ we have $\det g_v = -1$. Write $G_{2d}:=\SO(\bQ\otimes \Lambda_{2d})$. This is an algebraic group over $\bQ$.

\begin{lemma}The map
\[
	\cK \to G_{2d}(\bA_f),\,
	g \mapsto (\det g) \cdot g|_{\bA_f \otimes \Lambda_{2d}}
\]
has open image. If $d>1$, then the map is injective. If $d=1$ then it has a kernel of order $2$.
\end{lemma} 

\begin{proof}
The image is open, since it contains the open subgroup $\{ g \in \cK \mid \det g = 1\}$. Let $g$ be a non-trivial element of the kernel. Then $\det g = -1$, and  $g$ acts as $- 1$ on $\hat\bZ \otimes \Lambda_{2d}$ and hence as $-1$ on the discriminant module of $\Lambda_{2d}$. Since $g$ fixes $\lambda$, and $\Lambda$ is unimodular, it must act as $1$ on the discriminant module. Since the discriminant module is $\bZ/2d\bZ$, this is only possible if $d=1$. 

In the case $d=1$, the automorphism $\sigma$ of $\Lambda$ that acts as $-1$ on $\Lambda_2$ and as $1$ on $\lambda$ (which exists since the actions on the discriminant modules agree) defines the non-trivial element of the kernel.
\end{proof}

Let $\psi\colon \Lambda_{2d} \otimes \bR \to \bR$ be the quadratic form obtained by extending the pairing on $\Lambda_{2d}$ linearly, and let $\Omega^{\pm} \subset \Hom(\bS,G_{2d,\bR})$ be the period domain consisting of those weight $0$ Hodge structures of K3 type for which $\pm \psi$ is a polarization. It has two components, interchanged by the action of the two components of $G_{2d}(\bR) \cong \SO(2,19)$. 
 The pair $(G_{2d},\Omega^{\pm})$ is a Shimura datum, with reflex field $\bQ$.

On $\Sh_{\cK}^\an[G_{2d},\Omega^{\pm}]$ we have a tautological variation of $\bZ$-Hodge structures $\cH$,  defined group-theoretically as in \S~\ref{subsec:VHS} in terms of
\begin{enumerate}
\item the natural left action of $G_{2d}(\bQ)$ on $\Lambda \otimes \bQ$ (fixing $\lambda$),
\item the right action of $\cK$ on  $\Lambda \otimes\bQ$ given by $\det\colon \cK\to \{\pm1\} $
on $\Lambda_{2d} \otimes \bQ$ and trivial on $\bQ\lambda$,
\item the $\hat\bZ$-lattice $\Lambda \otimes \hat \bZ$.
\end{enumerate}

Note that the joint action of $\cK$ on $\Lambda\otimes \bA_f$, on the left via $\cK \to G_{2d}(\bA_f)$ and on the right as in (2) above, is the natural action of $\cK \subset \Orth(\Lambda\otimes \hat \bZ)$ on $\Lambda \otimes \bA_f$, and hence the lattice $\Lambda \otimes \hat \bZ$ is indeed invariant under this joint action.

\begin{proposition}\label{prop:Shimura-module-interpretation}
$\Sh_{\cK}^\an[G_{2d},\Omega^{\pm}]$ is the moduli stack of triples $(H,b,L)$ consisting of
\begin{enumerate}
\item a $\bZ$-VHS $H$ of weight $0$, with Hodge numbers $(1,20,1)$
\item a bilinear form $b\colon H\otimes H \to \bZ$
\item a section $L\in H$ of type $(0,0)$
\end{enumerate}
such that there exists locally an isometry $H\to \Lambda$ mapping $L$ to $\lambda$.
\end{proposition}

\begin{proof}
Since $\Lambda$ is the unique element of its genus, the tautological variation of Hodge structures $\cH$ (with the section induced by $\lambda$) satisfies the desired constraints. Conversely, given a triple $(H,b,L)$ over $S$, there exists, locally on $S$,
\begin{enumerate}
\item an isometry $\alpha\colon H\otimes \hat \bZ \isomto \Lambda \otimes \hat \bZ$ such that $\alpha(L)=\lambda$ and such that $\det\alpha$ restricts to an isomorphism $\delta \colon \det H \to \det \Lambda$ of rank one $\bZ$-modules
\item an isometry $\beta\colon H\otimes \bQ \isomto \Lambda \otimes \bQ$ such that $\beta(L)=\lambda$ and such that $\det\beta = \delta \otimes \bQ$
\end{enumerate}
The isometry $\beta$ defines by transport of structure a Hodge structure on $\Lambda \otimes \bR$ and hence an element $h\in \Omega^\pm$. The restriction of the composition $\beta \alpha^{-1} \colon \Lambda \otimes \hat \bA_f \to \Lambda \otimes \hat \bA_f$  to $\Lambda_{2d} \otimes \bA_f$ has determinant $1$ by construction, and hence defines an element of $g\in G_{2d}(\bA_f)$. Therefore, locally on $S$ the triple $(H,b,L)$ defines a pair $(h,g) \in \Omega^\pm \times G_{2d}(\bA_f)$. 

The choice of $\alpha$ is unique up to a unique element of $\cK$, and fixing $\alpha$, the choice of $\beta$ is unique up to a unique element of $G_{2d}(\bQ)$. It follows that the pairs $(h,g)$ glue to a section of the quotient stack $\Sh_{\cK}^\an[G_{2d},\Omega^\pm]$ which is unique up to unique isomorphism.
\end{proof}

\begin{proposition}\label{prop:K3-Shimura}
The map $\iota\colon \cM_{2d,\bC} \to \Sh_{\cK}^\an[G_{2d},\Omega^\pm]$ defined by mapping a polarized K3 surface $(\pi\colon X\to S, \cL)$ to the variation of Hodge structures $H:=\rR^2 \pi_\ast \bZ(1)$ and the section $L:=c_1(\cL)$ is an open immersion.
\end{proposition}

\begin{proof} In the light of Proposition \ref{prop:Shimura-module-interpretation}, this is a restatement of the Torelli theorem for polarized K3 surfaces, see \cite[\S~6.4 \& Thm~7.5.3]{HuybrechtsK3}.
\end{proof}

\begin{remark}Note that it is not the primitive cohomology
\[
	P := c_1(\cL)^\perp \subset \rH^2(X,\bQ(1))
\]
of a polarized K3 surface $(X,\cL)$, but rather its determinant twist $P \otimes \det P$ which corresponds to the standard tautological variation of Hodge structures on the orthogonal Shimura stack $\Sh_\cK[G_{2d},\Omega^\pm]$. 
\end{remark}

\begin{theorem}\label{thm:K3-Shimura}
The map $\iota$ descends to an open immersion
\[
\iota\colon \cM_{2d} \to \Sh_{\cK}[G_{2d},\Omega^\pm]
\]
defined over $\bQ$.
\end{theorem}

\begin{proof}
This follows from Deligne's theorem on absolute Hodge cycles \cite{Deligne82,Andre96a,Andre96b}, see the arguments of \cite[Cor.~5.4]{MadapusiPera15}. Alternatively, one can use the CM arguments of \cite{Rizov10} to show this for K3 surfaces with level structure, and deduce the result by passing to the quotient. Note however, that some care is needed in dealing with level structures, see Remark \ref{rmk:rizov} below.
\end{proof}

\begin{remark}
To the best of our knowledge, the basic result of Proposition \ref{prop:K3-Shimura}, describing  the stack of polarized K3 surfaces as an open substack of a Shimura stack seems to be missing from the literature, even for $d>1$ (so that $\cK$ is a subgroup of $G_{2d}(\bA_f)$). Madapusi Pera \cite[Proof of~5.3]{MadapusiPera15} gives a closely related statement: an open immersion between \'etale $2:1$ covers of $\cM_{2d}$ and $\Sh_{\cK}[G_{2d},\Omega^\pm]$; see also Remark \ref{rmk:MP2} below.
\end{remark}

\begin{remark}
If $d=1$ then the global stabilizer of $\Sh_{\cK}[G_2,\Omega^\pm]$ coming from the non-trivial element of the kernel of $\cK \to G_{2}(\bA_f)$ has a geometric interpretation: if $(X,\cL)$ is a K3 surface of degree $2$ then $|\cL|$ defines a $2:1$ map to a projective plane which induces a canonical involution on $(X,\cL)$. 
\end{remark}

Theorem \ref{thm:K3-Shimura} has the following corollary (see also \cite[Lem.~8.4.1]{Andre96a}):

\begin{corollary}\label{cor:orientability}
Let $\pi\colon X\to S$ a polarized family of K3 surfaces over a scheme $S$ over $\bQ$, and let $\bar s$ a geometric point of $S$. Then 
\begin{enumerate}
\item the action of $\pi_1(S,\bar s)$ on  $\det \rH^2_\et(X_{\bar s},\hat\bZ)$ has image in $\{\pm 1 \} \subset \hat\bZ^\times$
\item if $S$ is connected and the action of $\pi_1(S,\bar s)$ on $\det \rH^2_\et(X_{\bar s},\hat\bZ)$ is trivial, then there exists an isomorphism
\[
	\nu \colon \det \rR^2 \pi_\ast \hat\bZ(1) \longisomto \hat\bZ
\]
such that for every $s \in S(\bC)$ the map $\nu$ restricts to an isomorphism
\[
	\nu_s^\an \colon \det\rH^2(X_s,\bZ) \longisomto \bZ
\]
of free $\bZ$-modules of rank $1$. \qed
\end{enumerate}
\end{corollary}

\begin{remark}
The fact that the action of $\pi_1(S,s)$ on $\rH^2_\et(X_{\bar s},\bQ_\ell(1))$ respects the intersection pairing implies that the induced action on $\det \rH^2_\et(X_{\bar s},\bQ_\ell(1))$ takes values in $\{\pm 1\} \subset \bQ_\ell^\times$. Thus, we obtain for every $\ell$ a quadratic character
\[
	\chi_\ell\colon \pi_1(S,s) \to \{ \pm 1\}.
\]
The first assertion in the statement of the corollary is that $\chi_\ell$ is independent of $\ell$. This is in fact true for the determinant of the middle-cohomology of any even-dimensional proper smooth $X\to S$, see \cite[Lemma 3.2]{Saito12}.
\end{remark}

If $S$ is connected and the residue fields of $S$ can be embedded in $\bC$, then an isomorphism $\nu$ satisfying the condition on $\bC$-points in part (2) of Corollary \ref{cor:orientability} is unique up to sign. We will call such an isomorphism $\nu$ an \emph{orientation} of $X\to S$.

For $\pi\colon X\to S$ a family over an arbitrary $S$ over $\bQ$, we say that an isomorphism
\[
	\nu \colon \det \rR^2 \pi_\ast \hat\bZ(1) \longisomto \hat\bZ
\]
is an \emph{orientation} if for every $s\in S$ there is a finitely generated subfield $K \subset \kappa(s)$ such that $(X_s,\nu_s)$ descends to a pair $(X_K,\nu_K)$ with $\nu_K$ an orientation. Since every family $X\to S$ of polarized K3 surfaces over $\bQ$ can be defined over a scheme of finite type over $\bQ$, we see that the orientations on $X\to S$ form an \'etale $\{\pm1\}$-torsor.

Orientations will play an important role in the discussion of level structures below.  

\subsection{Level structures}\label{subsec:level-structures}

\begin{definition}\label{def:oriented-level-structure}
An  \emph{oriented full level $n$-structure} on a degree $2d$ polarized  K3 surface $(X,\cL)$ over a scheme $S$ over $\bQ$ is a pair $(\nu,\alpha)$ consisting of an orientation $\nu$ and
an isomorphism of \'etale sheaves
\[
	\alpha \colon \rR^2 \pi_\ast \bZ/n\bZ(1) \longisomto \Lambda \otimes \bZ/n\bZ
\]
such that
\begin{enumerate}
\item $\alpha(c_1(\cL)) = \lambda$,
\item $\alpha$ respects the pairings,
\item $\nu \otimes \bZ/n\bZ = \det \alpha$
\end{enumerate}
A \emph{full level $n$-structure} is an isomorphism of \'etale sheaves
\[
	\alpha \colon \rR^2 \pi_\ast \bZ/n\bZ(1) \longisomto \Lambda \otimes \bZ/n\bZ
\]
that \'etale locally on $S$ extends to an oriented full level $n$-structure $(\nu,\alpha)$.
\end{definition}

Note that if $n>2$ then $\nu$ is uniquely determined by $\alpha$ and hence a full level $n$-structure is the same as an oriented full level $n$-structure.

\begin{remark}\label{rmk:rizov}
Rizov \cite{Rizov06,Rizov10} defines a full level $n$-structure without the condition on orientations. However, this leads to problems and incorrect statements, that can essentially all be traced back to the fact that $\SO(\bA_f)$ is not of index $2$ in $\Orth(\bA_f)$ (the quotient is $\prod_v \{ \pm 1\}$). For example: with the definition of \cite[Def.~5.1.1]{Rizov06}, the collection of possible level $n$-structures on a fixed K3 surface over $\bC$ is infinite. In particular, \cite[Prop~2.4.6]{Rizov10} is incorrect as printed. (In the proof in \emph{loc.~cit.~}one cannot conclude that the element $q$ has determinant $1$.)
\end{remark}

Let $\cM_{2d,n}$ be the stack over $\bQ$ classifying polarized K3 surfaces of degree $2d$ and a full level $n$-structure, and let $\widetilde\cM_{2d,n}$ be the stack over $\bQ$ classifying polarized K3 surfaces with an oriented full level $n$-structure. 

Consider the profinite groups
\[
	\cK_{2d,n} := \{ g \in \Orth(\hat\bZ \otimes \Lambda) \mid
	\det g \in \{\pm 1\},\, g\lambda = \lambda,\, g \equiv 1 \bmod{n} \}
\]
and
\[
	\widetilde\cK_{2d,n} := \{ g \in \SO(\hat\bZ \otimes \Lambda) \mid
	 g\lambda = \lambda,\, g \equiv 1 \bmod{n} \}.
\]
Note that $\widetilde\cK_{2d,n} = \cK_{2d,n}$ for $n>2$.

The same arguments as in \S~\ref{subsec:K3-Shimura} yield the following generalization of Theorem \ref{thm:K3-Shimura}.

\begin{theorem}\label{thm:K3-level-Shimura}
For all $d\geq 1$ and $n\geq 1$ there are open immersions
\[
	\iota\colon \cM_{2d,n} \to \Sh_{\cK_{2d,n}}[G_{2d},\Omega^\pm]
\]
and
\[
	\iota\colon \widetilde\cM_{2d,n} \to \Sh_{\tilde\cK_{2d,n}}[G_{2d},\Omega^\pm]
\]
defined over $\bQ$, compatible with the natural forgetful maps and projections. \qed
\end{theorem}

\begin{remark}\label{rmk:MP2}
The statement about $\widetilde\cM_{2d,1}$ is already contained in the proof of Proposition~5.3 and in Corollary 5.4 in \cite{MadapusiPera15} (where our stack $\widetilde\cM_{2d,1}$ is denoted $\tilde{\mathrm{M}}^\circ_{2d}$). Note that on $\tilde\cM$ the local systems defined by the primitive cohomology and its determinant twist are isomorphic.

Rizov \cite[Prop.~2.4.6]{Rizov10} gives an open embedding of the moduli space of K3 surfaces with `full level $n$-structures' into a Shimura variety, but as indicated in Remark \ref{rmk:rizov}, the statement needs to be corrected by taking into account orientations, as in Definition \ref{def:oriented-level-structure} and Theorem \ref{thm:K3-level-Shimura}.
\end{remark}

\subsection{Models of K3 surfaces with complex multiplication}

We now come to the proof of Theorem \ref{thm:K3}. We use the notation of \S~\ref{subsec:intro-K3}, so $X/\bC$ is a K3 surface with complex multiplication, $T$ is the Mumford-Tate group of $\rH^2(X,\bQ(1))$. We denote by $C$ the centralizer of $T$ in $\Orth(\rH^2(X,\bQ(1)))$. 
\begin{lemma}Let $\pi$ be a subgroup of $\Orth(\rH^2(X,\hat\bZ(1)))$. Then the following are equivalent:
\begin{enumerate}
\item $\pi$ preserves $\Pic X \subset \rH^2(X,\hat\bZ(1))$ and acts on $V_{X,\bA_f}$ via $T(\bA_f)$,
\item $\pi \subset C(\bQ)T(\bA_f)$.
\end{enumerate}
\end{lemma}

\begin{proof}
Note that $V_{X,\bQ} \cong E$ with the evident action of $T \subset \Res^E_\bQ\bG_{m,E}$. Hence
he centralizer of $T$ in $\GL(V_{X,\bQ})$ is $E^\times$, and the centralizer of $T$ in $\Orth(V_{X,\bQ})$ is $T$. We deduce that  $C = \Orth((\Pic X)\otimes \bQ) \times T$ as subgroups of $\Orth(\rH^2(X,\bQ(1)))$.
The lemma now follows easily from the observation that $\Pic X = \rH^2(X,\hat\bZ(1)) \cap (\Pic X) \otimes \bQ$.
\end{proof}

\begin{proposition}\label{prop:K3-polarizability}
Let $\pi \subset C(\bQ)T(\bA_f) \cap \Orth(\rH^2(X,\hat\bZ(1)))$ be a compact subgroup. 
Then the following are equivalent:
\begin{enumerate}
\item there is a primitive ample $L \in \Pic X \subset \rH^2(X,\hat\bZ(1))$ invariant under $\pi$
\item $\pi$ preserves the ample cone $K_X \subset (\Pic X)\otimes\bR$.
\end{enumerate}
\end{proposition}

\begin{proof}
By compactness, $\pi$ acts with finite image on $\Pic X$.
The ample cone $K_X$ is the unique connected component of 
\[
	\big\{ x \in (\Pic X)\otimes \bR  \mid 
	\text{ $x^2>0$, and $x\cdot \delta \neq 0$ for all $\delta \in \Pic X$ with $\delta^2=-2$ } \big \}
\]
containing any ample class. The group $\pi$ permutes the set of connected components, and if $L \in K_X$ is fixed by $\pi$, then clearly the ample cone is preserved. 
 
 Conversely, assume that the ample cone is preserved. Then, since the ample cone is convex, taking the average over an orbit of any ample class will yield an ample class $L \in K_X \cap (\Pic X)\otimes \bQ$ that is fixed by $\pi$. Scaling if necessary, we may take $L$ to be primitive in $\Pic X$. 
\end{proof}

\begin{proof}[Proof of Theorem \ref{thm:K3}]
With the help of Proposition \ref{prop:K3-polarizability}, the proof follows essentially the same reasoning as the proof of Theorem \ref{thm:AV}. 
 
 If $L$ is a primitive ample line bundle of degree $2d$ on $X$, then identify the pair $(\rH^2(X,\bZ(1),L)$ with
 $(\Lambda,\lambda)$ and consider the Shimura datum $(G_{2d},\Omega^\pm)$ of \S~\ref{subsec:K3-Shimura}. The pair $(X,L)$ defines a complex point $x=(h,1)$ on
 $\Sh_{\cK}[G_{2d},\Omega^\pm]$. If $(\cX,\cL)$ is a model of $(X,L)$ over $F$ corresponding to a point $\xi$ under $x$, then by Theorem \ref{thm:K3-level-Shimura} the Galois representations $\rho \colon \Gal_F \to \cK$ and $\Gal_F \to \Aut \rH^2(\cX_{\bar F}, \hat \bZ(1))$ coincide.
 
Now if $\cX$ is a model of $X$ over $F$, then there exists a primitive ample $\cL \in \Pic_{X/F}(F)$, and $(\cX,\cL)$ defines a point on $\Sh_{\cK}[G_{2d},\Omega^\pm]$ with $\cL\cdot \cL = 2d$. By Theorem \ref{thm:Sh} its associated Galois representation $\rho \colon \Gal_F \to \cK$ satisfies properties (1) and (2) in Theorem \ref{thm:K3}. Since $\Gal_F$ fixes $\cL$, it fixes the ample cone in $(\Pic X)\otimes_\bZ\bR$ and we see that $\rho$ also satisfies (3).

Conversely, assume that
\[
	\rho\colon \Gal_F \to C(\bQ)T(\bA_f) \cap \Orth(\rH^2(X,\hat \bZ(1)))
\]
fixes the ample cone.  Then by Proposition \ref{prop:K3-polarizability}, it fixes some primitive ample $L$ with $L\cdot L=2d$ for some $d$. The pair $(X,L)$ defines a point $x$ on $\Sh_{\cK}[G_{2d},\Omega^\pm]$ with $L\cdot L = 2d$. The group scheme $T$ acts with determinant $1$ on $V_{X,\bQ}$ and since $\rho(\Gal_F)$ preserves $\Pic X$, it acts with determinant $\pm 1$ on $(\Pic X)\otimes \bA_f$. It follows that $\rho$ has image in $\cK \subset \Orth(\rH^2(X, \hat \bZ(1)))$, and by Theorem \ref{thm:Sh} the representation $\rho\colon \Gal_F \to \cK$ corresponds to an $F$-point $\xi$ under $x$, and hence to a model $(\cX,\cL)$ of $(X,L)$.
\end{proof}

\subsection*{Acknowledgments} I thank Olivier Benoist, Wessel Bindt, and Ben Moonen for fruitful discussions and comments on an earlier draft.

\bibliographystyle{plain}
\bibliography{taelman}

\end{document}